\documentclass[11pt]{amsart}
\usepackage{amsfonts}
\usepackage{amssymb, amstext, amscd, amsmath}
\usepackage{amsthm}
\usepackage{times}
\usepackage{txfonts}
\usepackage{indentfirst}
\usepackage[all]{xy}
\usepackage{subfigure}
\usepackage{graphicx}
\newtheorem{thm}{Theorem}[section]

\newtheorem{lem}[thm]{Lemma}

\theoremstyle{definition}
\newtheorem{rem}[thm]{Remark}

\numberwithin{equation}{section}
\numberwithin{equation}{section}
\def\square{\vbox{
      \hrule height 0.4pt
      \hbox{\vrule width 0.4pt height 5.5pt \kern 5.5pt \vrule width 0.4pt}
      \hrule height 0.4pt}}

\def\ch\mathrm{c h}

\begin{document}
\baselineskip 17.5pt

\newcommand{\auths}[1]{\textrm{#1},}
\newcommand{\artTitle}[1]{\textsl{#1},}
\newcommand{\jTitle}[1]{\textrm{#1}}
\newcommand{\Vol}[1]{\textbf{#1}}
\newcommand{\Year}[1]{\textrm{(#1)}}
\newcommand{\Pages}[1]{\textrm{#1}}

\title[surface sum of two handlebodies]{A characteristics for a surface sum  of two handlebodies along an annulus or a once-punctured torus to be a handlebody}

\author{Fengchun Lei$^*$, He Liu, Fengling Li$^\dag$}
\address{School of Mathematical Sciences, Dalian University of Technology, Dalian 116024, China}
\email{fclei@dlut.edu.cn, 1033410708@qq.com, fenglingli@dlut.edu.cn}

\author{Andrei Vesnin$\ddag$}
\address{Regional Mathematical Center, Department of Mechanics and Mathematics, Tomsk State University, Tomsk  634050, Russia}
\email{vesnin@math.nsc.ru}

\thanks{$*$Supported in part by a grant (No.11431009) of NSFC}

\thanks{$\dag$ Supported in part by grants (No.11671064) of NSFC and a grant of the Fundamental Research Funds for the Central Universities (No. DUT19LK15)}

\thanks{$\ddag$ Supported in part by the grant from Ministry of Science and Education of Russia (state agreement No.1.13557.2019/13.1).}


\begin{abstract}
The main results of the paper is that we give a characteristics for an annulus sum and a once-punctured torus sum of two handlebodies to be a handlebody as follows: 1. The annulus sum $H=H_1\cup_A H_2$ of two handlebodies $H_1$ and $H_2$ is a handlebody if and only if the core  curve of $A$ is a longitude for either $H_1$ or $H_2$. 2. Let $H=H_1\cup_T H_2$ be a surface sum of two handlebodies $H_1$ and $H_2$ along a once-punctured torus $T$. Suppose that $T$ is incompressible in both $H_1$ and $H_2$.  Then $H$ is a handlebody if and only if the there exists a collection $\{\delta, \sigma\}$ of simple closed curves on $T$ such that either $\{\delta, \sigma\}$ is primitive in $H_1$ or $H_2$, or $\{\delta\}$ is primitive in $H_1$  and $\{\sigma\}$ is primitive in $H_2$.
\end{abstract}
\maketitle

\section{Introduction}
Let $M_1$ and $M_2$ be two compact connected orientable 3-manifolds, $F_i\subset \partial M_i$ a compact connected surface, $i=1,2$, and $h:F_1\rightarrow F_2$ a homeomorphism. We call the 3-manifold $M=M_1\cup_h M_2$, obtained by gluing $M_1$ and $M_2$ together via $h$, a {\em surface sum} of $M_1$ and $M_2$. Denote $F=F_1=F_2$ in $M$. We also call $M$ a surface sum  of $M_1$ and $M_2$ along $F$, and denote it by $M=M_1\cup_F M_2$. When $F_i$ is a boundary component of $M_i$, $i=1,2$, $M$ is called an amalgamated 3-manifold of $M_1$ and $M_2$ along $F=F_1=F_2$.

Heegaard distances and related topics of amalgamation of two Heegaard splittings have been studied extensively in recent years, see, for example, ~\cite{Bach, KQ, La, LiT, Schu-Weid, YL}, etc.

In ~\cite{Lei1}, some properties of an annulus sum of 3-manifolds was obtained. In particular, a sufficient condition for an annulus sum of two handlebodies was given. In ~\cite{LiT}, some facts on Heegaard splittings of an annulus sum of 3-manifolds have been given, which played an essential role in calculating the Heegaard genus of the corresponding 3-manifold. Hyperbolic geometric structures related to quasi-Fuchsian realizations of once-punctured torus group were studied in~\cite{MPV}.

The main results of the paper is that we give a characteristics for an annulus sum and a once-punctured torus sum of two handlebodies to be a handlebody as follows.

\begin{thm}\label{thm1}
Let $H=H_1\cup_A H_2$ be a surface sum of two handlebodies $H_1$ and $H_2$ along an annulus $A$. Then $H$ is a handlebody if and only if the core curve of $A$ is a longitude for either $H_1$ or $H_2$.

\end{thm}

\begin{thm}\label{thm2}
Let $H=H_1\cup_T H_2$ be a surface sum of two handlebodies $H_1$ and $H_2$ along a once-punctured torus $T$. Suppose that $T$ is incompressible in both $H_1$ and $H_2$.  Then $H$ is a handlebody if and only if there exists a collection $\{\delta, \sigma\}$ of simple closed curves on $T$ such that \begin{itemize}
\item either $\{\delta, \sigma\}$ is primitive in $H_1$ or $H_2$,
\item or $\{\delta\}$ is primitive in $H_1$, and $\{\sigma\}$ is primitive in $H_2$.
\end{itemize}
\end{thm}

The paper is organized as follows. In section 2, we briefly introduce some definitions and preliminaries, and show some lemmas which will be used in the proofs of Theorem ~\ref{thm1} and Theorem ~\ref{thm2}. In section 3, we complete the proofs of Theorem ~\ref{thm1} and Theorem ~\ref{thm2}.

\section{Preliminaries}
The concepts and terminologies not defined in the paper are
standard, referring to, for example, ~\cite{He}, ~\cite{Jaco1}, or~\cite{S}.

An embedded 2-sphere $S$ in a 3-manifold $M$ is {\em trivial} if $S$ bounds a 3-ball in $M$. $S$ is {\em essential} if it is not trivial in $M$. If $M$ contains an essential 2-sphere, it is {\em reducible}. Otherwise, it is irreducible.

A properly embedded disk $D$ in a 3-manifold $M$ is {\em trivial} if $\partial D$ bounds a disk $D'$ in $\partial M$, and $D\cup D'$ bounds a 3-ball in $M$. $D$ is {\em essential} if it is not trivial in $H$. A properly embedded simple arc $\alpha$ in a surface $F$ is {\em trivial} if $\alpha$ cuts out of a disk from $F$. $\alpha$ is {\em essential} if it is not trivial in $F$.

Let $M$ be a compact 3-manifold. Suppose $F$ is a sub-surface of $\partial M$ or a surface properly embedded in $M$. If either
\begin{itemize}
\item[(1)] $F$ is a trivial disk, or
\item[(2)] $F$ is a trivial 2-sphere, or
\item[(3)] there is a disk $D\subset{M}$ such that $D\cap{F}=\partial{D}$ and $\partial{D}$ is an essential loop in~$F$,
\end{itemize}
then we say that $F$ is {\em compressible
}in $M$. Such a disk $D$ is called a {\em compressing disk}. We say
that $F$ is {\em incompressible} in $M$ if $F$ is not compressible
in $M$. If $\partial{M}$ is incompressible, then $M$ is said to be
{\em $\partial$-irreducible}. If $F$ is an incompressible surface in
$M$ and not parallel to a sub-surface of $\partial{M}$, then $F$ is
an {\em essential} surface in $M$.

A properly embedded surface $F$ in a 3-manifold $M$ is {\em $\partial$-compressible} (boundary-compressible) if either
\begin{itemize}
\item[(1)] $F$ is a trivial disk in $M$, or
\item[(2)] $F$ is not a disk and there exists a disk $D\subset M$ such that
$D\cap F=\alpha$ is an arc in $\partial D$, $D\cap \partial M=\beta$ is an arc in $\partial D$, with $\alpha\cap \beta= \partial \alpha=\partial\beta$ and $\alpha\cup \beta= \partial D$, and $\alpha$ is essential on $F$. $D$ is also called a $\partial$-compression disk of $F$.
\end{itemize}
$F$ is {\em $\partial$-incompressible} if it is not  $\partial$-compressible in $M$.

A {\em handlebody} $H$ of genus $n$ is a 3-manifold such that there exists a collection ${\mathcal D}=\{D_1,\cdots,D_n\}$ of $n$ pairwise disjoint properly embedded disks in $H$ and the manifold obtained by cutting $H$ open along $\mathcal D$ is a 3-ball. We call $\mathcal D$ a complete disk system for $H$. A {\em longitude} of the handlebody $H$ is a simple closed curve on $\partial H$ which intersects the boundary of an essential disk of $H$ in one point.

Let $H$ be a handlebody of genus $n$, and ${\mathcal J}=\{J_1,\cdots,J_p\}$ a collection of simple closed curves on $\partial H$. We say that ${\mathcal J}$ is {\em primitive} if $[J_1],\cdots,[J_p]\in \pi_1(H)$ (after some conjugations) can be extended to a generator set of $\pi_1(H)$. It is clear that a simple closed curve on $\partial H$ is primitive if and only if it is a longitude of $H$.

\begin{lem}\label{lem1}
Let $H$ be a handlebody of genus $n$, and ${\mathcal J}=\{J_1,\cdots,J_p\}$ ($p<n$) a collection of simple closed curves on $\partial H$. Suppose that there exists a collection ${\mathcal D}=\{D_1,\cdots,D_p\}$ of pairwise disjoint disks properly embedded in $H$ such that $|J_i\cap \partial D_i|=1$ for $1\leq i\leq p$, and $|J_i\cap \partial D_j|=0$ for $1\leq i\neq j\leq p$. Then ${\mathcal J}$ is  primitive.
\end{lem}

\begin{proof}
It is well-known that $\pi_{1} (H)$ is a free group on $n$ generators. A choice of a base-point and a complete system $\Delta = \{ \Delta_{1}, \ldots \Delta_{n} \}$ of oriented meridian disks determines a presentation of $\pi_{1} (H)$, namely, for any based oriented loop $\alpha$ (with the base point disjoint from $\Delta$ and in general position with $\Delta$) in $H$, start from the base point, go along $\alpha$ in the orientation, write down $x_{i} $ every time the loop passes through the disk $\Delta_{i}$ in a direction consistent with its normal orientation and $x_{i}^{-1}$ if the direction is not consistent. Thus, any simple closed curve $J$ on $\partial H$, when viewed as a (conjugate class) in $\pi_{1} (H)$ will be presented as a word in generators $x_{1}, \ldots, x_{n}$. Now consider a collection ${\mathcal J}=\{J_1,\cdots,J_p\}$ of curves from the statement. Without loss of generality we can assume that  a complete collection $\Delta$ is such that $\Delta_{i} = D_{i}$ for $i=1, \ldots, p$. Thus, a curve $J_{i}$ will be determined by a word $y_{i}$ in $\pi_{1} (H)$ which contains only once letter $x_{i}^{\varepsilon_{i}}$, with $\varepsilon_{i} = \pm 1$, and doesn't contain letters $x_{k}$ where $k \in \{1, 2, \ldots, p \} \setminus \{i\}$:
$$
y_{i} = u_{i} (x_{p+1}, \ldots, x_{n}) x_{i}^{\varepsilon_{i}} v_{i} (x_{p+1}, \ldots, x_{n}),
$$
where $u_{i}$ and $v_{i}$ are words in generators $x_{p+1}, \ldots, x_{n}$. Without loss of generality we can assume that curves $J_{i}$ are oriented in such a way that $\varepsilon_{i} = 1$, $i=1, \ldots, p$. Hence for $i = 1, \ldots, p$ we have
$$
 x_{i} = u_{i}^{-1} (x_{p+1}, \ldots, x_{n}) y_{i} v_{i}^{-1} (x_{p+1}, \ldots, x_{n}).
$$
For $i=p+1. \ldots, n$ we define $y_{i} = x_{i}$.
Since sets $\{ x_{1}, x_{2}, \ldots, x_{n}\}$ and $y_{1},y_{2},  \ldots, y_{n}$ are related by Nilsen transformations~\cite{MKS},
we conclude that $y_{1}, y_{2}, \ldots, y_{n}$ are generators of $\pi_{1} (H)$. Hence the collection $\mathcal J$ is primitive.
\end{proof}

\begin{lem}\label{lem2}
Let $M=M_1\cup_F M_2$ be a surface sum of two irreducible 3-manifolds $M_1$ and $M_2$ along $F$. Suppose that $F$ is incompressible in both $M_1$ and $M_2$. Then $M$ is irreducible.
\end{lem}

\begin{proof}
See~\cite{Lei1}.
\end{proof}

The following lemma ~\ref{lem3} is a well-known fact, for a proof, refer to \cite{Jaco3}.

\begin{lem}\label{lem3}
Let $M$ be a compact irreducible 3-manifold with non-empty boundary. If $\pi_{1}(M)$ is free, then $M$ is a handlebody.
\end{lem}

\begin{lem}\label{lem4}
Let $H$ be a handlebody of genus $n$, and $J$ a simple closed curve on $\partial H$. Suppose that $\partial H - J$ is incompressible in $H$. Let $H_J$ be the 3-manifold obtained by attaching a 2-handle to $H$ along $J$. Then
\begin{itemize}
\item either $H_J$ has incompressible boundary,
\item or $H_J$ is a 3-ball,
\end{itemize}
in the latter case, $H$ is a solid torus, and $J$ is a longitude for $H$.
\end{lem}

Lemma \ref{lem4} was first proved by Przytycki \cite{Przy} in 1983, then was generalized to the well known Handle Addition Theorem by Jaco \cite{Jaco2} in 1984.

\begin{lem}\label{lem5}
Let $H$ be a surface sum of handlebodies $H_{1}$ and $H_{2}$ along a once-punctured torus $T$. Suppose that $T$ is incompressible in both $H_1$ and $H_2$. If there exists a collection $\{\delta, \sigma\}$ of simple closed curves on $T$ such that
\begin{itemize}
\item  either $\{\delta, \sigma\}$ is primitive in $H_1$ or $H_2$,
\item or $\{\delta\}$ is primitive in $H_1$, and $\{\sigma\}$ is primitive in $H_2$,
\end{itemize}
then $H$ is a handlebody.
\end{lem}

\begin{proof}
An example of a collection $\{ \delta, \sigma\}$ is presented in Fig.~\ref{fig1}.

\begin{figure}[!htb]
\centering
  \includegraphics[width=5cm]{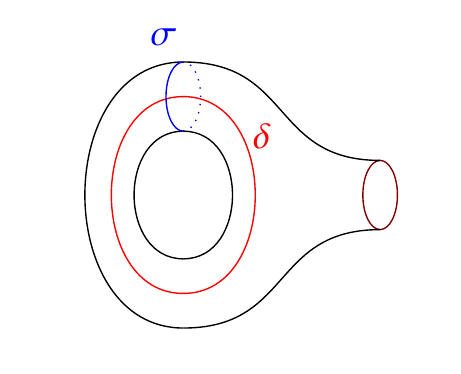}
  \caption{A collection $\{ \delta, \sigma \}$.} \label{fig1}
\end{figure}
 First assume that $\delta$ is a primitive curve in $H_{1}$ and $\sigma$ is a primitive curve in $H_{2}$.
Let $n_{1} = g(H_{1}) \geq 2$ be genus of $H_{1}$ and $n_{2} = g(H_{2}) \geq 2$ be genus of $H_{2}$. Then there exist elements $x_{1}, \ldots x_{n_{1}-1} \in \pi_{1} (H_{1})$ such that $\{ [\delta], x_{1}, \ldots, x_{n_{1}-1} \}$ is the set of free generators of $\pi_{1}(H_{1})$. Similarly, there exist elements $y_{1}, \ldots, y_{n_{2}-1} \in \pi_{1} (H_{2})$ such that $\{ [\sigma], y_{1}, \ldots, y_{n_{2}-1} \}$ is the set of free generators of $\pi_{1} (H_{2})$. An amalgamation product of two handlebodies induces the amalgamation  product of their fundamental groups. Denote the natural inclusions $i : T \to H_{1}$ and $j :  T \to H_{2}$. They induces homomorphisms
$i_{*} : \pi_{1} (T) \to \pi_{1} (H_{1})$ and $j_{*} : \pi_{1} (T)  \to \pi_{1} (H_{2} )$.   Then by the Seifert~-- van Kampen theorem (see, for example~\cite[Theorem~70.2]{Munkres})
the kernel of the homomorphism
$$
\phi : \pi_{1} (H_{1}) * \pi_{1} (H_{2}) \to \pi_{1} (H)
$$
is the least normal subgroup $N$ of the free product that contains elements represented by words of the form $i_{*} (\sigma)^{-1} j_{*} (\sigma)$ and $i_{*} (\delta)^{-1} j_{*} (\delta)$.  Since $\sigma$ and $\delta$ are free generators of $\pi_{1} (T)$ we get that $\pi_{1} (H)$ is a free group with free generators $x_{1}, \ldots, x_{n_{1}-1}, y_{1}, \ldots, y_{n_{2}-1}$. In virtue of Lemma \ref{lem3}, $H$ is a handlebody.

The proof of the other case is similar.
\end{proof}

\begin{thm}\label{thm0}
Let $H$ be a handlebody of genus $n$, and $J$ a simple closed curve on $\partial H$. Let $H_J$ be the 3-manifold obtained by attaching a 2-handle to $H$ along $J$. Then $H_J$ is a handlebody if and only if $J$ is a longitude for $H$.
\end{thm}

\begin{proof} The sufficiency is clear.

We will prove the necessity.  Assume that $H_J$ is a handlebody. $J$ can not bound a disk in $H$. Otherwise, $H_J$ is reducible, contradicting to that $H_J$ is a handlebody. If $\partial H - J$ is compressible in $H$, make the maximal compressions of $\partial H - J$ in $H$ to get a handlebody $H'$ such that $\partial H' - J$ is incompressible in $H'$. It is clear that $H_J$ is a handlebody if and only if that $H_J'$ is a handlebody. It follows from Lemma \ref{lem4} that either $H_J'$ has incompressible boundary, or $H_J'$ is a 3-ball, in the latter case, $H'$ is a solid torus, and $J$ is a longitude for $H'$. Now $H_J'$ is a handlebody, so only the latter case happens. Thus $J$ is also a longitude for $H$. This completes the proof.
\end{proof}

\section{The proofs of the main results}
A simple closed curve $\alpha$ on an annulus $A$ is called a \emph{core} curve of $A$ if $\alpha$ is parallel to one boundary component of $A$.

We now come to the proofs of the main results.

\vskip 3mm\noindent
{\bf Proof of Theorem ~\ref{thm1}:} First assume that the core curve $\alpha$ of $A$ is a longitude for $H_2$, say. Let $D$ be an essential disk of $H_2$ such that $\alpha$ intersects $\partial D$ in one point. Let $T$ be a regular neighborhood of $\alpha\cup \partial D$ in $\partial H_2$. Then $T$ is a once-punctured torus, and $\beta=\partial T$ bounds a separating disk $\Delta$ in $H_2$, which cuts $H_2$ into a solid torus $H_2'$ ($T\cup \Delta=\partial H_2'$) and a handlebody $H_2''$ of genus $g(H_2)-1$. We may assume that $\beta\cap A=\emptyset$. Thus $H=H_1\cup_A H_2=H_1 \cup_A (H_2'\cup_\Delta H_2'')=(H_1 \cup_A H_2')\cup_\Delta H_2'' \cong H_1 \cup_\Delta H_2''$, where it is clear that $H_1 \cup_A H_2'\cong H_1$, so $H$ is a handlebody of genus $g(H_1)+g(H_2)-1$.

Now assume that $H=H_1\cup_A H_2$ is a handlebody. If the core curve $J$ of $A$ bounds a disk $D$ in $H_1$, say, let $N=D\times I$ be a regular neighborhood of $D$ in $H_1$ such that $N\cap \partial H_1=A$. Then $H=H_1\cup_A H_2=\overline{H_1-N}\cup (H_2\cup_A N)$. $H$ is a handlebody, so $H_2\cup_A N$ is a handlebody. By Theorem ~\ref{thm0}, $J$ is a longitude for $H_2$. Next we assume that $J$ bounds no disk in neither $H_1$ nor $H_2$.

Let ${\mathcal D}=\{D_1,\cdots,D_n\}$ ($n=g(H_1)+g(H_2)-1$) be a complete disk system for $H$ such that ${\mathcal D}$ and $A$ are in general position, and ${\mathcal D}\cap A=\bigcup_{i=1}^n D_i\cap A$ has minimal components over all such complete disk systems for $H$.

{\bf Case 1}. ${\mathcal D}\cap A=\emptyset$.

Then the core curve of $A$ is a longitude for either $H_1$ or $H_2$. The conclusion holds.

In fact, since $H$ is a handlebody of genus $n=g(H_1)+g(H_2)-1$, if we cut open $H$ along ${\mathcal D}$, we get a 3-ball. On the other hand, if we cut open $H_i$ along ${\mathcal D}\cap H_i$ to get $K_i$, then either $K_1$ is a solid torus and $K_2$ is a 3-ball, or $K_2$ is a solid torus and $K_1$ is a 3-ball. In both cases, the conclusion holds.

{\bf Case 2}. ${\mathcal D}\cap A\ne\emptyset$.

Each component of ${\mathcal D}\cap A$ is either a simple circle in the interior of some disk in ${\mathcal D}$, or a simple arc properly embedded in some disk in ${\mathcal D}$.

{\bf Claim 1}. No component of ${\mathcal D}\cap A$ is a circle.

Otherwise, let $\alpha$ be an innermost circle component of ${\mathcal D}\cap A$, say $\alpha\subset D_i$. $\alpha$ bounds a disk $E$ on $D_i$ with $A\cap int(E)=\emptyset$. $\alpha$ cannot be a core curve of $A$. Thus, $\alpha$ also bounds a disk in $A$. Let $\beta$ be a circle component of ${\mathcal D}\cap A$ which is innermost on $A$. Then $\beta$ bounds a disk $E'$ on $A$ with ${\mathcal D}\cap int(E')=\emptyset$. Say $\beta\subset D_j$. $\beta$ bounds a disk $E''$ in $D_j$. Set $S=E'\cup E''$. $S$ is a 2-sphere in $H$, so it bounds a 3-ball $B$ in $H$. We can use the 3-ball $B$ to make an isotopy to push $E''$ to pass through $E'$.  After the isotopy, the complete disk system ${\mathcal D}$ is changed to ${\mathcal D}'$. Clearly, ${\mathcal D}'\cap A$ has at least one component ($\beta$) less than that of ${\mathcal D}\cap A$. This contradicts to minimality of ${\mathcal D}\cap A$. So Claim 1 holds.

Thus, each component of ${\mathcal D}\cap A$ is a simple arc. Let $\gamma$ be an outermost arc component of ${\mathcal D}\cap A$, say, $\gamma$ is a properly embedded arc in $D_k$ in ${\mathcal D}$. $\gamma$ cuts out of a disk $\Sigma$ from $D_k$ with $A\cap int(\Sigma)=\emptyset$. $\gamma$ is also a properly embedded arc in $A$.

{\bf Claim 2}. The two end points of $\gamma$ can not lie in the same boundary component of $A$.

Otherwise, let $\gamma'$ be such an arc component of ${\mathcal D}\cap A$ which is outermost on $A$. Then $\gamma'$ cuts out of a disk $\Sigma'$ from $A$ with ${\mathcal D}\cap int(\Sigma')=\emptyset$. Say $\gamma'\subset D_l\in {\mathcal D}$. $\gamma'$ cuts $D_l$ into two disks $D_l'$ and $D_l''$. Set $D_{l1}=\Sigma'\cup D_l'$ and $D_{l2}=\Sigma'\cup D_l''$, and ${\mathcal D}'=\{D_1,\cdots,D_{l-1}, D_{l1}, D_{l+1},\cdots,D_n\}$, ${\mathcal D}''=\{D_1,\cdots,D_{l-1}, D_{l2}, D_{l+1},\cdots,D_n\}$. Then at least one of ${\mathcal D}'$ and ${\mathcal D}''$ is a complete disk system for $H$, say ${\mathcal D}'$, and after a small isotopy, at least the component $\gamma'$ has been removed, ${\mathcal D}'\cap A$ has at least one component less than that of ${\mathcal D}\cap A$. This again contradicts to minimality of ${\mathcal D}\cap A$. So Claim 2 holds.

Now let $\lambda$ be an outermost arc component of ${\mathcal D}\cap A$ whose two end points are lying in the different boundary components of $A$, say, $\lambda$ is a properly embedded arc in $D_m$ in ${\mathcal D}$. $\lambda$ cuts out of a disk $\Lambda$ from $D_m$ with $A\cap int(\Lambda)=\emptyset$. Then $\Lambda$ is a properly embedded disk in either $H_1$ or $H_2$. It is obvious that the core curve of $A$ intersects $\partial \Lambda$ in exact one point. Thus, the core curve of $A$ is a longitude for either $H_1$ or $H_2$.

This completes the proof of Theorem ~\ref{thm1}. \hfill$\Box$

\vskip 3mm\noindent
{\bf Proof of Theorem ~\ref{thm2}:}
Assume that $H=H_1\cup_T H_2$ is a handlebody. Set $n_i=g(H_i)$, $i=1,2$. Since $T$ is incompressible in both $H_1$ and $H_2$, $n_i\geq 2$, $i=1,2$. Thus $n=g(H)=n_1+n_2-2\geq 2$.

Let ${\mathcal D}=\{D_1,\cdots,D_n\}$ be a complete disk system for $H$ such that ${\mathcal D}$ and $T$ are in general position, and ${\mathcal D}\cap T=\bigcup_{i=1}^n D_i\cap A$ has minimal components over all such complete disk systems for $H$.

{\bf Claim 1}. ${\mathcal D}\cap A\ne\emptyset$.

Otherwise, after cutting open $H$ along ${\mathcal D}$, we get a 3-ball $B$ containing $T$. $T$ is compressible in $B$, so it is compressible in $H_1$ or $H_2$, a contradiction.

Each component of ${\mathcal D}\cap T$ is either a simple circle in the interior of some disk in ${\mathcal D}$, or a simple arc properly embedded in some disk in ${\mathcal D}$.

{\bf Claim 2}. No component of ${\mathcal D}\cap A$ is a circle.

Otherwise, let $\alpha$ be an innermost circle component of ${\mathcal D}\cap T$, say $\alpha\subset D_i$. $\alpha$ bounds a disk $E$ on $D_i$ with $A\cap int(E)=\emptyset$. Say $E\subset H_1$. Since $T$ is incompressible in $H_1$, $\alpha$ also bounds a disk in $T$. Let $\beta$ be a circle component of ${\mathcal D}\cap A$ which is innermost on $T$. Then $\beta$ bounds a disk $E'$ on $T$ with ${\mathcal D}\cap int(E')=\emptyset$. Say $\beta\subset D_j$. $\beta$ bounds a disk $E''$ in $D_j$. Set $S=E'\cup E''$. $S$ is a 2-sphere in $H$, so it bounds a 3-ball $B$ in $H$. We can use the 3-ball $B$ to make an isotopy to push $E''$ to pass through $E'$.  After the isotopy, the complete disk system ${\mathcal D}$ is changed to ${\mathcal D}'$. Clearly, ${\mathcal D}'\cap A$ has at least one component ($\beta$) less than that of ${\mathcal D}\cap A$. This contradicts to minimality of ${\mathcal D}\cap A$. So Claim 2 holds.

Thus, each component of ${\mathcal D}\cap T$ is a simple arc properly embedded in $T$ and some disk in ${\mathcal D}$. Let $\gamma$ be an outermost arc component of ${\mathcal D}\cap A$, say, $\gamma$ is a properly embedded arc in $D_k$ in ${\mathcal D}$. $\gamma$ cuts out of a disk $\Delta$ from $D_k$ with $T\cap int(\Delta)=\emptyset$. Say $\Delta\in H_1$.

{\bf Claim 3}. $\gamma$ is an essential arc in $T$.

Otherwise, let $\gamma'$ be such an arc component of ${\mathcal D}\cap T$ which is outermost on $T$. Then $\gamma'$ cuts out of a disk $\Delta'$ from $T$ with ${\mathcal D}\cap int(\Delta')=\emptyset$. Say $\gamma'\subset D_l\in {\mathcal D}$. $\gamma'$ cuts $D_l$ into two disks $D_l'$ and $D_l''$. Set $D_{l1}=\Sigma'\cup D_l'$ and $D_{l2}=\Sigma'\cup D_l''$, and ${\mathcal D}'=\{D_1,\cdots,D_{l-1}, D_{l1}, D_{l+1},\cdots,D_n\}$, ${\mathcal D}''=\{D_1,\cdots,D_{l-1}, D_{l2}, D_{l+1},\cdots,D_n\}$. Then at least one of ${\mathcal D}'$ and ${\mathcal D}''$ is a complete disk system for $H$, say ${\mathcal D}'$, and after a small isotopy, at least the component $\gamma'$ has been removed, ${\mathcal D}'\cap T$ has at least one component less than that of ${\mathcal D}\cap T$. This again contradicts to minimality of ${\mathcal D}\cap T$. So Claim 3 holds.

Thus, $\gamma$ is non-separating in $T$, so $\Delta$ is an non-separating essential disk in $H_1$. Denote $\gamma''=\overline{\partial \Delta -\gamma}$. Then $\gamma''$ is a arc lying in $\partial H_1$ (as well as in $\partial H$).

Let $N=\Delta\times I$ be a regular neighborhood of $\Delta$ in $H_1$ with $\Delta=\Delta\times \frac{1}{2}$ such that $N\cap T=\gamma\times I$ and $N\cap \partial H=\gamma''\times I$. Set $H_1'=\overline{H_1-N}$, $H_2'=H_2\cup_{\gamma\times I} N$, and $A=\overline{T-\gamma\times I}\bigcup \Delta\times 0\bigcup \Delta\times 1$. In fact, $A$ is the result of doing a $\partial$-compression on $T$ in $H$ along $\Delta$, see Figure 2 below.

\begin{figure}[!htb]
\centering
  \includegraphics[width=0.6\textwidth]{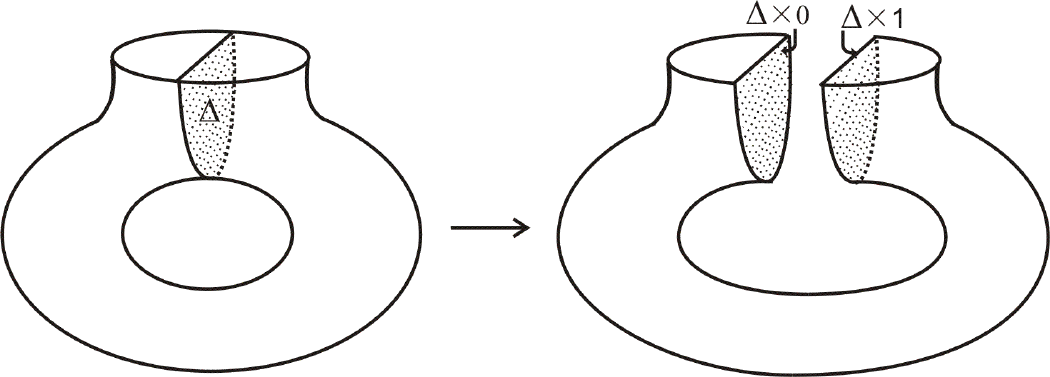}
  \caption{$\partial$-compression along $\Delta$}\label{figure3}
\end{figure}
Thus $H_1'$ is a handlebody of genus $n_1-1$, $H_2'$ is a handlebody of genus $n_2$, $A$ is an annulus, and $H=H_1'\cup_A H_2'$ is an amalgamation of $H_1'$ and $H_2'$ along $A$.

By Theorem \ref{thm1}, the core curve $\sigma$ is a longitude for either $H_1'$ or $H_2'$. If $\sigma$ is a longitude for $H_1'$, there exists an essential disk $\Sigma$ in $H_1'$ such that $|\sigma\cap\partial \Sigma|=1$. It is clear that we may assume that $\sigma$ is disjoint from $\Delta$, and $\Sigma$ is disjoint from $\Delta$. We now choose a simple closed curve $\delta$ (see Figure 3 below) such that $|\delta\cap\partial \Delta|=1$, and $|\delta\cap\sigma|=1$. $\{[\delta],[\sigma]\}\subset\pi_1(T)$ is a generator set of $\pi_1(T)$. Thus, $\{\delta, \sigma\}$ is primitive in $H_1$.

\begin{figure}[!htb]
\centering
  \includegraphics[width=0.4\textwidth]{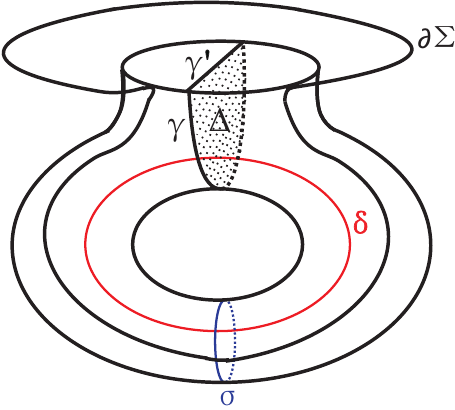}
  \caption{The curves $\sigma$, $\delta$, $\partial\Sigma$, and the arcs $\gamma$, $\gamma'$.} \label{figure3}
\end{figure}

Similarly, if $\sigma$ is a longitude for $H_2'$ (therefore $H_2$), then $\{\delta\}$ is primitive in $H_1$, and $\{\sigma\}$ is primitive in $H_2$.

The proof of the other direction follows directly from Lemma \ref{lem5}.


This completes the proof of Theorem ~\ref{thm2}. \hfill$\Box$

\begin{rem}
We only consider the surface sum $H=H_1\cup_F H_2$ of two handlebodies along a surface $F$, where $F$ is an annulus or a once-punctured torus. In general case, it maybe reduced to the case that the gluing surface in the surface sum of two handlebodies is not connected, in which the computation of the fundamental group of the corresponding 3-manifold is difficult.
\end{rem}




\bibliographystyle{amsplain}

\end{document}